\documentclass[12pt, a4paper]{article}

\usepackage{tikz}
\usepackage{xcolor}
\usepackage{eso-pic}

\usepackage{graphicx}
\usepackage{float}

\usepackage[colorlinks,linkcolor=red,anchorcolor=blue,citecolor=green]{hyperref}

\usepackage{amsmath}
\usepackage{amssymb}
\usepackage{mathrsfs}
\usepackage{bm}
\usepackage{extarrows}
\usepackage{amsthm}
\usepackage{color}
\usepackage{booktabs}
\usepackage{multirow}
\usepackage{array}
\usepackage{comment}

\usepackage{listings}
\usepackage{bbm}

\definecolor{dkgreen}{rgb}{0,0.6,0}
\definecolor{gray}{rgb}{0.5,0.5,0.5}
\definecolor{mauve}{rgb}{0.58,0,0.82}
\usepackage[a4paper,margin=2.5cm]{geometry}
\usepackage{ulem}

\newcommand{\Ecal}{\mathcal{E}}
\newcommand{\Xcal}{\mathcal{X}}

\newcommand{\Pcal}{\mathcal{P}}

\newcommand{\Lcal}{\mathcal{L}}
\newcommand{\Wcal}{\mathcal{W}}

\renewcommand{\d}{\ensuremath{\mathrm{d}}}
\newcommand{\dt}{ \ensuremath{\mathrm{d} t } }
\newcommand{\dx}{ \ensuremath{\mathrm{d} x} }

\newcommand{\p}{ \mathbb{P}}

\newcommand{\e}{ \operatorname{\mathbb E}}

\newcommand{\Real}{\mathbb{R}}
\newcommand{\ZZ}{\mathbb{Z}}

\newcommand{\abs}[1]{\left\vert#1\right\vert}
\newcommand{\norm}[1]{\left\|#1\right\|}

\newcommand{\floor}[1]{\left \lfloor #1 \right \rfloor }

\newcommand{\transpose}{\textsf{T}} 

\newtheorem{re}{Remark}[section]
\newtheorem{lemma}{Lemma}[section]
\newtheorem{theo}{Theorem}[section]
\newtheorem{coro}{Corollary}[section]

\newtheorem{assum}{Assumption}[section]
\newtheorem{prop}{Proposition}[section]

\numberwithin{equation}{section}

\title{Empirical approximation to invariant measures of
	non-degenerate McKean--Vlasov dynamics}
\author{Wenjing Cao\thanks{School of Mathematical Sciences, 
		Fudan University, 
		Shanghai 200433, China (email: {\tt wjcao22@m.fudan.edu.cn}). The author's research is supported
        by the National Key R{\&}D Program of China (No.~2022ZD0116401)} 
        \and 
        Kai Du\thanks{SCMS, Fudan University, Shanghai, China;
    Shanghai Artificial Intelligence Laboratory
    (email: {\tt kdu@fudan.edu.cn}). 
    The author's research is supported
    by the National Natural Science Foundation of China (No.~12222103),
    and by the National Key R{\&}D Program of China (No.~2022ZD0116401)
	}}
\date{\today}
\begin{document}
	\maketitle
	\begin{abstract}
		This paper studies the approximation of invariant measures of McKean--Vlasov dynamics with non-degenerate additive noise. While prior findings necessitated a strong monotonicity condition on the McKean--Vlasov process, we expand these results to encompass dissipative and weak interaction scenarios. Utilizing a reflection coupling technique, we prove that the empirical measures of the McKean--Vlasov process and its path-dependent counterpart can converge to the invariant measure in the Wasserstein metric. The Curie--Weiss mean-field lattice model serves as a numerical example to illustrate empirical approximation. 
	\end{abstract}

    \section{Introduction}\label{Introduction}

    This paper is concerned with the approximation of the invariant measure of the following McKean--Vlasov process with non-degenerate additive noise:
	\begin{equation}\label{sde1}
		\d X_t=b\big(X_t, \Lcal(X_t)\big)\dt +\d B_t
	\end{equation}
    where $B$ is a standard Brownian motion in $\Real^d$, and $\Lcal(X_t)$ denotes the distribution law of $X_t$. 
    The McKean--Vlasov dynamics arise from the ``propagation of chaos'' of particle systems with mean-field interaction (cf. \cite{kac1956foundations, mckean1967propagation, mckean1966class}), having wide application in a variety of fields such as physics, economics and computer sciences (cf. \cite{lasry2007mean, 10.1214/19-AAP1499, sznitman1991topics}). A probability measure $\mu^*$ is called invariant for \eqref{sde1} if $P^*_t\mu^*=\mu^*$ for all $t\geq0$, where $(P_t^*)_{t\geq0}$ denotes the nonlinear transition semigroup such that $P_t^*\Lcal(X_0):=\Lcal(X_t)$. 
    
    This work serves as a seamless continuation of the recent paper \cite{du2023empirical}, where the authors introduced an innovative approach for approximating the invariant measure of the general McKean--Vlasov process governed by
	\begin{equation}\label{sde_origin}
		\d X_t=b\big(X_t, \Lcal(X_t)\big)\dt +\sigma\big(X_t, \Lcal(X_t)\big)\d B_t.
	\end{equation}
    The key idea is to introduce a path-dependent stochastic differential equation (SDE)
    \[
    	\d Z_t=b\Big(Z_t, \frac{1}{t}\int_{0}^{t}\delta_{Z_s} \d s\Big)\dt +\sigma\Big(Z_t, \frac{1}{t}\int_{0}^{t}\delta_{Z_s} \d s\Big)\d B_t,
    \]
    and argue that the empirical measure of $Z$, namely, $\mathcal{E}_t(Z):= \frac{1}{t}\int_{0}^{t}\delta_{Z_s} \d s$,
    can converge to the invariant measure of the McKean--Vlasov process $X$ in \eqref{sde_origin}
    under specific conditions.
    In \cite{du2023empirical} the authors proved this convergence by employing a strong monotonicity condition: there are
    constants $\alpha > \beta \ge 0$ such that
    \[
        2\langle b(x,\mu)-b(y,\nu),x-y \rangle + \|\sigma(x,\mu)  -\sigma(y,\nu)\|^{2}
	\leq -\alpha\abs{x-y}^{2}+\beta \mathcal{W}_{2}(\mu,\nu)^{2}
    \]
    for all \(x,\,y\in\Real^{d}\) and \(\mu,\,\nu\in \mathcal{P}_2(\Real^d)\) (refer to notations at the end of this section).
    However, such a condition rules out many useful models from applications; 
    for instance, the Curie--Weiss mean-field model
    \begin{equation}\label{Curie0}
        \d X_t= \left[-\beta(X_t^3-X_t)+\beta K\e[X_t]\right]\d t+\d B_t
    \end{equation}
    does not satisfy the strong monotonicity if $\beta>0$.
    On the other hand, it is known that the McKean--Vlasov process governed by \eqref{sde1} has a unique invariant measure under certain dissipativity conditions along with the non-degeneracy nature of the equation (see Assumption \ref{dissi}; cf. \cite[Remarks 2.8 and 5.1]{main}).
    Hence, it naturally raises the question of whether the findings presented in \cite{du2023empirical} can be extended to encompass the dissipative case.
    
    Building upon the idea in \cite{du2023empirical}, we study the empirical approximation to the invariant measure for the McKean--Vlasov SDE \eqref{sde1} by imposing dissipativity and weak interaction conditions. The convergence is characterized by upper bound estimates for the $1$-Wasserstein distance $\Wcal_1$:
    \begin{equation}\label{simple}    \e\big[\Wcal_1(\Ecal_t(X),\mu^*)\big],\, \e\big[\Wcal_1(\Ecal_t(Z),\mu^*)\big]=O(t^{-\varepsilon})
    \end{equation}
    for $\varepsilon>0$, 
    where $\mu^*$ denotes the invariant measure of \eqref{sde1}, and $Z$ is the path-dependent counterpart of \eqref{sde1}:
    \begin{equation}\label{sde3}
        \d Z_t=b\big(Z_t,\Ecal_t(Z)\big)\d t+\d B_t,\quad Z_0=X_0.
    \end{equation}
    Furthermore, we also consider weighted empirical measures as detailed in Section \ref{weight_section}. 

    Taking advantage of the non-degeneracy of \eqref{sde1} we employ a reflection coupling method (cf.~\cite{eberle2016reflection}) to compare \eqref{sde1} and \eqref{sde3} with their associated Markovian SDE
	\begin{equation}\label{sde2}
	 	\d Y_t=b(Y_t,\mu^*)\d t+\d B_t,\quad Y_0=X_0.
	 \end{equation}
    Another crucial aspect of our analysis involves establishing a Wasserstein convergence estimate for \eqref{sde2}.
    Notably, existing results in \cite{wang2019limit} and \cite{wang2022wasserstein} do not appear to be directly applicable to our case. 
    In light of this, we employ a technique introduced in \cite{riekert2022convergence} to establish convergence in the $1$-Wasserstein distance for the empirical measures of \eqref{sde2} (see Theorem \ref{WeightedMarkovConvergence}). 
    A similar convergence result was previously obtained in \cite{riekert2022convergence} for exponentially contractive discrete-time Markov chains through a combination of Kantorovich duality and Fourier analysis.
	 
    The processes described by \eqref{sde3} are commonly referred as self-interacting diffusions in the literature. 
    In contrast to previous studies on this topic (cf.~\cite{benaim2002self,10.1214/EJP.v17-2121}, etc.), our research provides a quantitative characterization of the convergence of empirical measures (also known as occupation measures in the literature). We also remark that we extend our analysis to encompass a broader range of equation forms and do so without necessitating symmetry or convexity conditions on the coefficients.

    The rest of our paper is organized as follows. 
    Section \ref{results} states the main assumptions and theorems. In Sections \ref{Proof} and \ref{WeightedProof}, we engage in the rigorous proofs of Theorems \ref{DistributionDependent}, \ref{PathDependent} and \ref{WeightedPathDependent} respectively. 
    Section \ref{appendix} studies the convergence in Wasserstein distance for empirical measures of exponentially contractive Markov processes, which can directly lead to Corollary \ref{WeightedMarkovSDE} and thus lays a foundation for main theorems in Section \ref{results}. In Section \ref{Experiments}, numerical results are demonstrated.

    We conclude the introduction by clarifying some notations used in this paper. 
    We denote by $\abs{\cdot}$ and $\langle\cdot,\cdot\rangle$ the Euclidean norm and the inner product respectively in $\Real^d$,  and by $\norm{\cdot}$ the Frobenius norm of a matrix. Let $(\Xcal,\rho)$ be a complete separable metric space, and we denote by $\Pcal(\Xcal)$ the space of all Borel probability measures on $\Xcal$. For any $\alpha\in(0,1]$ and function $f:\Xcal\to\mathbb{C}$, the $\alpha$-${\rm H\ddot{o}lder}$ constant of $f$ is denoted by $\abs{f}_{\alpha}:=\sup\limits_{x\neq y}\frac{\abs{f(x)-f(y)}}{\rho(x,y)^{\alpha}}$; specifically, $\abs{\cdot}_1$ denotes the Lipschitz constant. For $p\geq 1$, the $p$-Wasserstein distance between $\mu,\,\nu\in\Pcal(\Xcal)$ is defined as 
	 $$
	 \Wcal_p(\mu,\nu):=\inf\Big\{\Big(\e\big[\rho(\xi,\zeta)^p\big]\Big)^{1/p}:\, \Lcal(\xi)=\mu,\, \Lcal(\zeta)=\nu\Big\}.$$
	 $\Pcal_p(\Xcal)$ denotes the space of all $\mu\in\Pcal(\Xcal)$ with finite $p$-th moment, i.e. $\rho(\cdot,y)\in L^p(\Xcal,\mu)$ for all $y\in\Xcal$; specifically, $\Pcal_p(\Real^d):=\{\mu\in\Pcal(\Real^d):\, \int_{\Real^d}\abs{x}^p\mu(\dx)<\infty\}$. $\Pcal_p(\Xcal)$ becomes a complete metric space if equipped with the distance $\Wcal_p$; for $p=1$, the $\Wcal_1$ distance between any $\mu,\,\nu\in\Pcal_1(\Xcal)$ is characterized by Kantorovich duality (cf. \cite{optimal}): 
    $$
    \Wcal_1(\mu,\nu)=\sup\limits_{\abs{f}_1\leq1}\abs{\mu(f)-\nu(f)}.$$ 

    \section{Main results}\label{results}
    This section comprises two parts, dealing with unweighted and weighted empirical measures, respectively. 
    
    \subsection{Convergence of empirical measures}\label{uniform}
        First and foremost, we assume the dissipativity of $b(\cdot,\cdot)$:
	\begin{assum}\label{dissi}
		For the continuous mapping $b:\Real^d\times\Pcal_1(\Real^d)\to\Real^d$, there exists an increasing function $\kappa:(0,\infty)\to\Real$ satisfying $\lim\limits_{r\to\infty}\kappa(r)=:\kappa_{\infty}>0,\, \lim\limits_{r\to 0^+}r\kappa(r)=0$ and that 
		$$
		\big\langle x-y,b(x,\mu)-b(y,\mu)\big\rangle\leq -\kappa(\abs{x-y})\abs{x-y}^2$$
		for all $x,\,y\in\Real^d$ and $\mu\in\Pcal_1(\Real^d)$.
	\end{assum}
	
	After $\kappa(\cdot)$ is defined, we import an auxiliary function $f\in C^2[0,\infty)$ such that
	\begin{equation}\label{auxiliary}
			\left\{
			\begin{aligned}
				f^{\prime}(r)&=\frac{1}{2}\int_{r}^{\infty}s \exp\Big(-\frac{1}{2}\int_{r}^{s}\tau\kappa(\tau)\d \tau\Big)\d s,\quad \forall r\geq0,\\
			f(0)&=0	.
			\end{aligned}
			\right.
	\end{equation}
	Throughout the rest of this paper except Section \ref{appendix}, the term $f$ refers to the function defined by \eqref{auxiliary}.

	We then give the second assumption, the weak interaction property of $b(\cdot,\cdot)$:
	\begin{assum}\label{weaki}
		There exists a constant $0<\eta< \frac{1}{f^{\prime}(0)}$ such that
		$$
		\abs{b(x,\mu)-b(x,\nu)}\leq \eta \Wcal_1(\mu,\nu)$$
		for all $x\in\Real^d$ and $\mu,\,\nu\in\Pcal_1(\Real^d)$.
	\end{assum}
	\begin{re}
		In the original definition of $\kappa(\cdot)$ in Assumption \ref{dissi}, $\kappa_{\infty}$ doesn't have to be finite. In fact, even if $\kappa_{\infty}=\infty$, we can truncate $\kappa(\cdot)$ to $\kappa^L(\cdot):=\kappa(\cdot)\wedge L$ for some $L\in(0,\infty)$. This increases the value of $f^{\prime}(0)$, but Assumption \ref{weaki} still holds for the same $\eta$ as long as $L$ is large enough. In this paper, unless specifically stated, $\kappa_{\infty}$ is regarded as finite.
	\end{re}
	
	Assumptions \ref{dissi} and \ref{weaki} ensure the existence and uniqueness of the invariant measure $\mu^*\in\Pcal_1(\Real^d)$, since \eqref{sde1} is time homogeneous (cf. \cite[Remarks 2.8 and 5.1]{main}).
 
    For the Markovian SDE \eqref{sde2}, the uniform $L^q$ bound of its solution is required, which can be fulfilled by controlling the growth rates of $b(\cdot,\cdot)$ (cf. \cite[Remark 2.7]{main}):
	\begin{assum}\label{Lq}
		For the solution $Y=(Y_t)$ of \eqref{sde2}, there exist constants $q>1$ and $C_q\in(0,\infty)$ such that 
		$$
		\sup\limits_{t\geq0}\e[\abs{Y_t}^q]\leq C_q.$$
	\end{assum} 
 
    Now we begin to state the main theorems of this paper. For a distribution-dependent SDE, we have the following result:
	\begin{theo}\label{DistributionDependent}
		Under Assumptions \ref{dissi}, \ref{weaki} and \ref{Lq}, for any solution $X$ of the distribution-dependent SDE \eqref{sde1}, we have 
		$$
		\e\big[\Wcal_1(\Ecal_t(X),\mu^*)\big]=O(t^{-\varepsilon}),$$
		where $0<\varepsilon<\min\{\frac{1}{d}(1-\frac{1}{q}),\frac{1}{2}(1-\frac{1}{q})\}$.
	\end{theo} 

	For a path-dependent SDE, the following estimate can be acquired:
	\begin{theo}\label{PathDependent}
		Under Assumptions \ref{dissi}, \ref{weaki} and \ref{Lq}, for any solution $Z$ of the path-dependent SDE \eqref{sde3}, we have
		$$
		\e\big[\Wcal_1(\Ecal_t(Z),\mu^*)\big]=O(t^{-\varepsilon}),$$
		where  $0<\varepsilon<\min\{\frac{1}{d}(1-\frac{1}{q}),\frac{1}{2}(1-\frac{1}{q}),1-\eta f^{\prime}(0)\}$.
	\end{theo} 

	\subsection{Convergence of weighted empirical measures}\label{weight_section}
	For weighted cases, we first define the set of weight families 
	\begin{equation}\label{weight0}
	\mathbf{\Pi}:=\{w=(w_t)_{t\geq0}:w_t\in\Pcal[0,1]\}.
        \end{equation}
	For any stochastic process $X$ and weight family $w\in\mathbf{\Pi}$, we denote by $$\Ecal_t^w(X)=\int_{0}^{1}\delta_{X_{ts}} w_t(\d s)$$ the weighted empirical measures of $X$ for $t\geq0$. For any $\varepsilon\in(0,1]$, we consider the following two sets of special weight families:
	\begin{equation}\label{weight}
		\begin{aligned}
			\mathbf{\Pi}_1(\varepsilon):&=\left\{w\in\mathbf{\Pi}: \limsup\limits_{t\to\infty}\int_{0}^{1}s^{-\varepsilon}w_t(\d s)<\frac{1}{\eta \kappa_{\infty}f^{\prime}(0)^2} \right\}\\
			\mathbf{\Pi}_2(\varepsilon):&=\left\{\overline{w}\in\mathbf{\Pi}: \limsup\limits_{t\to\infty}\int_{0}^{1}t^{\varepsilon}\wedge s^{-\varepsilon}\overline{w}_t(\d s)<\infty,\right. \\
			&\qquad \left.\limsup\limits_{t\to\infty}\int_{0}^{1}\int_{0}^{1}t^{\varepsilon}\wedge \abs{s_1-s_2}^{-\varepsilon}\overline{w}_t(\d s_1)\overline{w}_t(\d s_2)<\infty\right\}\\
			&=\left\{\overline{w}\in\mathbf{\Pi}: \sup\limits_{t\geq0}\int_{0}^{1}t^{\varepsilon}\wedge s^{-\varepsilon}\overline{w}_t(\d s)<\infty,\right. \\
			&\qquad \left.\sup\limits_{t\geq0}\int_{0}^{1}\int_{0}^{1}t^{\varepsilon}\wedge \abs{s_1-s_2}^{-\varepsilon}\overline{w}_t(\d s_1)\overline{w}_t(\d s_2)<\infty\right\}.
		\end{aligned}
	\end{equation}

	\begin{re}
		For the Lebesgue measure $\mathbf{\lambda}$, i.e. $\mathbf{\lambda}_t(\d s)=\d s$ for all $t\geq0$, we have $\mathbf{\lambda}\in\mathbf{\Pi}_1(\varepsilon_1)\cap\mathbf{\Pi}_2(\varepsilon_2)$ for all $0<\varepsilon_1<1-\eta\kappa_{\infty}f^{\prime}(0)^2$ and $0<\varepsilon_2<1$. $\Ecal_t^{\mathbf{\lambda}}(\cdot)$ is equivalent to unweighted empirical measures $\Ecal_t(\cdot)$. \\
            Another typical example is discrete sampling $w^{\tau}=(w^{\tau}_t)$ with a fixed step length $\tau>0$:
 	 $$
 	 	w^{\tau}_t=\begin{cases}
 	 		\frac{1}{n}\sum_{k=1}^{n}\delta_{k\tau/t} &n\tau \leq t< (n+1)\tau,\\
 	 		\delta_0 &0\leq t<\tau,
 	 	\end{cases}$$
  	 i.e. for $Z=(Z_t)$ ,
  	 $$
  	 	\Ecal_t^{w^{\tau}}(Z)=\begin{cases}
  	 		\frac{1}{n}\sum_{k=1}^{n}\delta_{Z_{k\tau}} &n\tau \leq t< (n+1)\tau,\\
  	 		\delta_{Z_0} &0\leq t<\tau,
  	 	\end{cases}$$
   	 where $n=n(t):=\floor{t/\tau}$. Similar to the Lebesgue measure, $w^{\tau}\in\mathbf{\Pi}_1(\varepsilon_1)\cap\mathbf{\Pi}_2(\varepsilon_2)$ for all $0<\varepsilon_1<1-\eta\kappa_{\infty}f^{\prime}(0)^2$ and $0<\varepsilon_2<1$, regardless of the step length.   
        \end{re}
	
	For $\varepsilon_1,\varepsilon_2\in(0,1]$, and $w\in\mathbf{\Pi}_1(\varepsilon_1)\cap\mathbf{\Pi}_2(\varepsilon_2),\,\overline{w}\in\mathbf{\Pi}_2(\varepsilon_2)$, we consider the weighted path-dependent SDE:
	\begin{equation}\label{weighted_sde3}
		\d Z_t=b\big(Z_t,\Ecal_t^w(Z)\big)\d t+\d B_t,\quad Z_0=X_0,
	\end{equation}
	trying to estimate $\e\big[\Wcal_1(\Ecal_t^{\overline{w}}(Z),\mu^*)\big]$. Here $\mu^*$ still denotes the invariant measure of \eqref{sde1}; $b(\cdot,\cdot)$ is similar to the one in \eqref{sde1}, which satisfies Assumption \ref{dissi} and a strengthened version of Assumption \ref{weaki}: 
	\begin{assum}\label{weaki_plus}
		There exists a constant $0<\eta<\frac{1}{\kappa_{\infty}f^{\prime}(0)^2}\leq \frac{1}{f^{\prime}(0)}$ such that
		$$
		\abs{b(x,\mu)-b(x,\nu)}\leq \eta \Wcal_1(\mu,\nu)$$
		for all $x\in\Real^d$ and $\mu,\,\nu\in\Pcal_1(\Real^d)$.
	\end{assum} 
 
	For a weighted path-dependent SDE, the following estimate can be acquired:
	\begin{theo}\label{WeightedPathDependent}
		For $w\in\mathbf{\Pi}_1(\varepsilon_1)\cap\mathbf{\Pi}_2(\varepsilon_2),\, \overline{w}\in\mathbf{\Pi}_2(\varepsilon_2)$ and the solution $Z$ of the weighted path-dependent SDE \eqref{weighted_sde3}, which satisfies Assumptions \ref{dissi}, \ref{weaki_plus} and \ref{Lq}, we have 
		$$
		\e\big[\Wcal_1(\Ecal_t^{\overline{w}}(Z),\mu^*)\big]=O\left(t^{-\varepsilon}\right),$$
		where
		$0<\varepsilon<\min\{\varepsilon_1,\frac{\varepsilon_2}{d}(1-\frac{1}{q}),\frac{\varepsilon_2}{2}(1-\frac{1}{q})\}$. 
	\end{theo} 
	
	\section{Proof of Theorems \ref{DistributionDependent} and \ref{PathDependent}}\label{Proof}
	\subsection{Distribution-dependent equations}
	We take smooth functions $\lambda,\,\pi:\Real^d\to\Real$ satisfying $\lambda(\cdot)^2+\pi(\cdot)^2=1$ and 
	$$
	\lambda(x)=\begin{cases}
		1, &\abs{x}\geq\delta,\\
		0, &\abs{x}\leq\delta/2.
	\end{cases}$$
	Here the parameter $0<\delta<1$ will be optimized later. Then we define a reflection coupling $(\hat{X},\hat{Y})$ such that
	$$
	\left\{
	\begin{aligned}
		\d \hat{X}_t&=b\big(\hat{X}_t,\Lcal(\hat{X}_t)\big)\d t+ \lambda(\Delta_t)\d B_t+\pi(\Delta_t)\d \hat{B}_t,\\
		\d \hat{Y}_t&=b(\hat{Y}_t,\mu^*)\d t+ \lambda(\Delta_t)(I_d-2e_t e_t^\transpose)\d B_t+\pi(\Delta_t)\d \hat{B}_t, 
	\end{aligned}
	\right.$$
	where $\hat{B}_t$ is a Brownian motion independent of $(B_t,X_0)$, $\Delta_t:=\hat{X}_t-\hat{Y}_t$, and 
	$$
	e_t:=\begin{cases}
		\frac{\Delta_t}{\abs{\Delta_t}},  &\Delta_t\neq0,\\
		0,  &\Delta_t=0.
	\end{cases}$$
    We set the initial data $\hat{X}_0=\hat{Y}_0=X_0=Y_0$, and write $\Lcal(X_0)=\gamma_0$. Applying Levy's characterization, we get $\Lcal(X)=\Lcal(\hat{X})$ and $\Lcal(Y)=\Lcal(\hat{Y})$.
	
	After some simple computation, we get 
	$$
	\begin{aligned}
		\d\Delta_t&=\Big(b\big(\hat{X}_t,\Lcal(\hat{X}_t)\big)-b(\hat{Y}_t,\mu^*)\Big)\d t+2\lambda(\Delta_t)e_t e_t^\transpose\d B_t\\
		&=\Big(b\big(\hat{X}_t,\Lcal(\hat{X}_t)\big)-b\big(\hat{Y}_t,\Lcal(\hat{X}_t)\big)\Big)\d t\\
            &\ +\Big(b\big(\hat{Y}_t,\Lcal(\hat{X}_t)\big)-b(\hat{Y}_t,\mu^*)\Big)\d t+2\lambda(\Delta_t)e_t e_t^\transpose\d B_t\\
		&:=(\Delta b_t^1+\Delta b_t^2)\d t+2\lambda(\Delta_t)e_t e_t^\transpose\d B_t,
	\end{aligned}$$
	where $\Delta b_t^1=b\big(\hat{X}_t,\Lcal(\hat{X}_t)\big)-b\big(\hat{Y}_t,\Lcal(\hat{X}_t)\big)$ and $\Delta b_t^2=b\big(\hat{Y}_t,\Lcal(\hat{X}_t)\big)-b(\hat{Y}_t,\mu^*)$.
	By the Ito--Tanaka formula, we have
	\begin{equation}\label{ito1.1}
		\d\abs{\Delta_t}=\frac{1}{2\abs{\Delta_t}}\Big(2\big\langle\Delta_t,\Delta b_t^1+\Delta b_t^2\big\rangle\Big)\d t+2\lambda(\Delta_t)e_t^\transpose\d B_t.
	\end{equation}
	By Assumptions \ref{dissi} and \ref{weaki}, we get
	\begin{equation}\label{ito1.2}
		\d\abs{\Delta_t}\leq-\kappa(\abs{\Delta_t})\abs{\Delta_t}\d t+\eta \Wcal_1\big(\Lcal(\hat{X}_t),\mu^*\big)\d t+2\lambda(\Delta_t)e_t^\transpose\d B_t.
	\end{equation}

        From \eqref{auxiliary}, we notice that $f$ is strictly increasing on $[0,\infty)$ and satisfies the following ordinary differential equation for all $r>0$:
	\begin{equation}\label{ODE}
	2f^{\prime\prime}(r)-r\kappa(r)f^{\prime}(r)=-r.
        \end{equation}
	By \cite[Lemma 5.1]{main}, we have $f^{\prime\prime}(r)\leq0$ and
	\begin{equation}\label{le1}
		\kappa_{\infty}^{-1}\leq f(r)/r\leq f^{\prime}(0) 
	\end{equation}
    for all $r>0$. Thus we combine \eqref{ito1.1} with \eqref{ito1.2}, apply Ito's formula again to $f(\abs{\Delta_t})$, and get 
	\begin{equation}\label{distribution_ito}
		\begin{aligned}
			\d f(\abs{\Delta_t})&\leq\big(-\kappa(\abs{\Delta_t})\abs{\Delta_t}f^{\prime}(\abs{\Delta_t})+2f^{\prime\prime}(\abs{\Delta_t})\lambda(\Delta_t)^2\big)\d t\\
			&\ +\eta f^{\prime}(\abs{\Delta_t})\Wcal_1\big(\Lcal(\hat{X}_t),\mu^*\big)\d t\\
			&\ +2f^{\prime}(\abs{\Delta_t})\lambda(\Delta_t)e_t^\transpose\d B_t.
		\end{aligned}
	\end{equation}
	By \eqref{ODE}, we have 
	$$
	-\kappa(\abs{\Delta_t})\abs{\Delta_t}f^{\prime}(\abs{\Delta_t})+2f^{\prime\prime}(\abs{\Delta_t})\lambda(\Delta_t)^2=-\abs{\Delta_t}-2f^{\prime\prime}(\abs{\Delta_t})\pi(\Delta_t)^2.$$
	We write $\rho(r):=\sup
 \limits_{x\in[0,r]}(-2f^{\prime\prime}(x))$, and get
	$$
	-\kappa(\abs{\Delta_t})\abs{\Delta_t}f^{\prime}(\abs{\Delta_t})+2f^{\prime\prime}(\abs{\Delta_t})\lambda(\Delta_t)^2\leq-\abs{\Delta_t}+\rho(\delta).$$
        By \cite[Remarks 2.8 and 5.1]{main}, we have 
        $$
	\Wcal_1\big(\Lcal(\hat{X}_t),\mu^*\big)=\Wcal_1\big(\Lcal(X_t),\mu^*\big)\leq \kappa_{\infty}f^{\prime}(0)\exp\Big(-\big(\frac{1}{f^{\prime}(0)}-\eta\big)t\Big)\Wcal_1(\gamma_0,\mu^*).$$
	
	Combining the above and taking expectation on both sides of \eqref{distribution_ito}, we have
	$$
	\begin{aligned}
		\frac{\d}{\d t}\e[f(\abs{\Delta_t})]&\leq -\e[\abs{\Delta_t}]+\rho(\delta)+\eta f^{\prime}(0)\Wcal_1\big(\Lcal(\hat{X}_t),\mu^*\big)\\
		&\leq -\frac{1}{f^{\prime}(0)}\e[f(\abs{\Delta_t})]+\rho(\delta)+\eta \kappa_{\infty}f^{\prime}(0)^2\exp\Big(-\big(\frac{1}{f^{\prime}(0)}-\eta\big)t\Big)\Wcal_1(\gamma_0,\mu^*).
	\end{aligned}$$
	By Gronwall's inequality, we get	
	\begin{equation}\label{gronwall01}
			\e[f(\abs{\Delta_t})]\leq f^{\prime}(0)\rho(\delta)+\kappa_{\infty}f^{\prime}(0)^2\Wcal_1\left(\gamma_0,\mu^*\right)\exp\Big(-\big(\frac{1}{f^{\prime}(0)}-\eta\big)t\Big).
	\end{equation}
    For empirical distributions $\Ecal_t(X)$ and $\Ecal_t(Y)$, using the definition of Wasserstein distance, we have
	\begin{equation}\label{empirical}
		\Wcal_1\big(\Ecal_t(\hat{X}),\Ecal_t(\hat{Y})\big)\leq \frac{1}{t}\int_{0}^{t}\abs{\Delta_s}\d s\leq \frac{\kappa_{\infty}}{t}\int_{0}^{t}f(\abs{\Delta_s})\d s.
	\end{equation}
    Thus, combining \eqref{empirical}, \eqref{gronwall01} and \eqref{le1} we get
    $$
    \begin{aligned}
    	\e\big[\Wcal_1(\Ecal_t(X),\Ecal_t(Y))\big]&=\e\left[\Wcal_1\big(\Ecal_t(\hat{X}),\Ecal_t(\hat{Y})\big)\right]\\
    	&\leq \frac{\kappa_{\infty}}{t}\int_{0}^{t}\e[f(\abs{\Delta_s})]\d s\\
    	&\leq\kappa_{\infty}f^{\prime}(0)\rho(\delta)+\frac{1}{t} \frac{\kappa_{\infty}^2f^{\prime}(0)^2\Wcal_1\left(\gamma_0,\mu^*\right)}{\frac{1}{f^{\prime}(0)}-\eta}.
    \end{aligned}$$
	Let $\delta\to0^+$, we have 
	\begin{equation}\label{distribution_final}
		\e\big[\Wcal_1(\Ecal_t(X),\Ecal_t(Y))\big]\leq \frac{\kappa_{\infty}^2f^{\prime}(0)^3\Wcal_1\left(\gamma_0,\mu^*\right)}{1-\eta f^{\prime}(0)}\frac{1}{t}=O(\frac{1}{t}).
	\end{equation}

	Combining \eqref{distribution_final} with Corollary \ref{WeightedMarkovSDE}, we finish the proof of Theorem \ref{DistributionDependent}.
	
	\subsection{Path-dependent equations}
	Just as we did in the previous section, we compare the path-dependent SDE \eqref{sde3} to the Markovian one \eqref{sde2}, and get a reflection coupling $(\tilde{Z},\tilde{Y})$:
	$$
	\left\{
	\begin{aligned}
		\d \tilde{Z}_t&=b\big(\tilde{Z}_t,\Ecal_t(\tilde{Z})\big)\d t+\lambda(V_t)\d B_t+\pi(V_t)\d \hat{B}_t,\\
		\d \tilde{Y}_t&=b(\tilde{Y}_t,\mu^*)\d t+ \lambda(V_t)(I_d-2u_t u_t^\transpose)\d B_t+\pi(V_t)\d \hat{B}_t, 
	\end{aligned}
	\right.$$
	where $\tilde{Z}_0=\tilde{Y}_0=Z_0=Y_0,\, V_t:=\tilde{Z}_t-\tilde{Y}_t$, and
	$$
	u_t:=\begin{cases}
		\frac{V_t}{\abs{V_t}},  &V_t\neq0,\\
		0,  &V_t=0.
	\end{cases}$$
	
	Similarly we get
	$$
	\begin{aligned}
		\d f(\abs{V_t})&\leq\big(-\kappa(\abs{V_t})\abs{V_t}f^{\prime}(\abs{V_t})+2f^{\prime\prime}(\abs{V_t})\lambda(V_t)^2\big)\d t\\
		&\ +\eta f^{\prime}(\abs{V_t})\Wcal_1\big(\Ecal_t(\tilde{Z}),\mu^*\big)\d t+2f^{\prime}(\abs{V_t})\lambda(V_t)e_t^\transpose\d B_t,
	\end{aligned}$$
	and then
	$$
	\begin{aligned}
		\frac{\d}{\d t}\e[f(\abs{V_t})]&\leq -\e[\abs{V_t}]+\rho(\delta)+\eta f^{\prime}(0)\e\big[\Wcal_1\big(\Ecal_t(\tilde{Z}),\mu^*\big)\big]\\
		&\leq -\e[\abs{V_t}]+\rho(\delta)+\eta f^{\prime}(0)\left(\e\big[\Wcal_1\big(\Ecal_t(\tilde{Z}),\Ecal_t(\tilde{Y})\big)\big]+\e\big[\Wcal_1\big(\Ecal_t(\tilde{Y}),\mu^*\big)\big]\right)\\
		&\leq -\e[\abs{V_t}]+\eta f^{\prime}(0)\int_{0}^{1}\e[\abs{V_{ts}}]\d s+\eta f^{\prime}(0)\e\big[\Wcal_1\big(\Ecal_t(\tilde{Y}),\mu^*\big)\big]+\rho(\delta).
	\end{aligned}$$
	We define functions $\varphi(t):=\e[\abs{V_t}]$ and $\psi(t):= \e[f(\abs{V_t})]$, which satisfy $$\varphi,\,\psi\geq0,\, \varphi(0)=\psi(0)=0,\, \kappa_{\infty}^{-1}\varphi\leq\psi\leq f^{\prime}(0)\varphi,$$
	and take $\delta=\delta(t)$ such that $\rho(\delta(t))=O(t^{-1})$ as $t\to\infty$.
	We use $C>0$ to represent a constant independent of $t,\,\varphi$ and $\psi$, whose value may vary through lines. Then by Corollary \ref{WeightedMarkovSDE}, for any $0<\varepsilon_0<\min\{\frac{1}{2}(1-\frac{1}{q}),\frac{1}{d}(1-\frac{1}{q})\}$, we have 
	$$
	\eta f^{\prime}(0)\e\big[\Wcal_1\big(\Ecal_t(Y),\mu^*\big)\big]+\rho(\delta(t))\leq C(1\wedge t^{-\varepsilon_0}).$$ 
	Therefore, we get
	\begin{equation}\label{ODI}
		\frac{\d}{\d t}\psi(t)\leq -\varphi(t)+\eta f^{\prime}(0)\int_{0}^{1}\varphi(ts)\d s+C(1\wedge t^{-\varepsilon_0}).
	\end{equation} 	
 
	Integrating both sides of \eqref{ODI} from $0$ to $t$, we have 
	$$
	0\leq \psi(t)\leq -\int_{0}^{t}\varphi(\tau)\d \tau+\eta f^{\prime}(0)\int_{0}^t\Big(\int_0^1\varphi(s_1 s_2)\d s_1\Big)\d s_2+C(t\wedge t^{1-\varepsilon_0}),$$
	and therefore 
	$$
	\int_{0}^{t}\varphi(\tau)\d \tau\leq \eta f^{\prime}(0)\int_{0}^t\Big(\int_0^1\varphi(s_1 s_2)\d s_1\Big)\d s_2+C(t\wedge t^{1-\varepsilon_0})$$
	for all $t>0$.
	We divide both sides by $t$, and get
	$$
	\int_{0}^{1}\varphi(ts)\d s\leq\eta f^{\prime}(0)\int_{0}^{1}\int_{0}^{1}\varphi(ts_1 s_2)\d s_1\d s_2+C(1\wedge t^{-\varepsilon_0})$$
	for all $t\geq0$.
	
	Now we fix $C$, and claim $\sup\limits_{t\geq0}t^{\varepsilon}\int_{0}^{1}\varphi(ts)\d s\leq \frac{C}{1-\varepsilon-\eta f^{\prime}(0)}$ for any given $0<\varepsilon<\min\{\varepsilon_0, 1-\eta f^{\prime}(0)\}$. Otherwise we let $T:=\inf\{t\geq0:t^{\varepsilon}\int_{0}^{1}\varphi(ts)\d s>\frac{C}{1-\varepsilon-\eta f^{\prime}(0)}\}$, and by continuity we have $T^{\varepsilon}\int_{0}^{1}\varphi(Ts)\d s= \frac{C}{1-\varepsilon-\eta f^{\prime}(0)}=\sup\limits_{0\leq t\leq T}t^{\varepsilon}\int_{0}^{1}\varphi(ts)\d s$. Thus we have 
	$$
	\begin{aligned}
		\frac{C}{1-\varepsilon-\eta f^{\prime}(0)}&=T^{\varepsilon}\int_{0}^{1}\varphi(Ts)\d s\leq \eta f^{\prime}(0)\int_{0}^{1}T^{\varepsilon}\int_{0}^{1}\varphi(Ts_1 s_2)\d s_1\d s_2+C\\
		&=\eta f^{\prime}(0)\int_{0}^{1}s_2^{-\varepsilon}\Big((Ts_2)^{\varepsilon}\int_{0}^{1}\varphi(Ts_2 s_1)\d s_1\Big)\d s_2+C\\
		&\leq \eta f^{\prime}(0)T^{\varepsilon}\int_{0}^{1}\varphi(T s_1)\d s_1\int_{0}^{1}s_2^{-\varepsilon}\d s_2+C\\
		&\leq \frac{\eta f^{\prime}(0)}{1-\varepsilon}\cdot\frac{C}{1-\varepsilon-\eta f^{\prime}(0)}+C\\
		&=\frac{C}{1-\varepsilon-\eta f^{\prime}(0)}\big(1-\varepsilon\frac{1-\varepsilon-\eta f^{\prime}(0)}{1-\varepsilon}\big)\\
		&<\frac{C}{1-\varepsilon-\eta f^{\prime}(0)},
	\end{aligned}$$
	which leads to contradiction. Therefore, the claim holds, and we get 
	\begin{equation}\label{path_final}
		\e\big[\Wcal_1\left(\Ecal_t(Z),\Ecal_t(Y)\right)\big]\leq \int_{0}^{1}\varphi(ts)\d s=O(t^{-\varepsilon}).
	\end{equation}
	for $0<\varepsilon<\min\{\frac{1}{d}(1-\frac{1}{q}), \frac{1}{2}(1-\frac{1}{q}), 1-\eta f^{\prime}(0)\}$. 
	
	Combining \eqref{path_final} with Corollary \ref{WeightedMarkovSDE}, we complete the proof of Theorem \ref{PathDependent}.
	
	\begin{re}\label{pointwise}
		If we combine the above claim with \eqref{ODI}, we may also obtain the pointwise estimate $\varphi(t)=O(t^{-\varepsilon})$ and $\psi(t)=O(t^{-\varepsilon})$, where $0<\varepsilon<\min\{\frac{1}{d}(1-\frac{1}{q}), \frac{1}{2}(1-\frac{1}{q}), 1-\eta f^{\prime}(0)\}$. 
	\end{re} 

	\section{Proof of Theorem \ref{WeightedPathDependent}}\label{WeightedProof}
	For the weighted path-dependent SDE \eqref{weighted_sde3} with $w\in\mathbf{\Pi}_1(\varepsilon_1)\cap\mathbf{\Pi}_2(\varepsilon_2)$,
	we still compare it to the Markovian SDE \eqref{sde2}, getting a reflection coupling $(\tilde{Z},\tilde{Y})$:
	$$
	\left\{
	\begin{aligned}
		\d \tilde{Z}_t&=b\big(\tilde{Z}_t,\Ecal_t^w(\tilde{Z})\big)\d t+ \lambda(V_t)\d B_t+\pi(V_t)\d \hat{B}_t,\\
		\d \tilde{Y}_t&=b(\tilde{Y}_t,\mu^*)\d t+ \lambda(V_t)(I_d-2u_t u_t^\transpose)\d B_t+\pi(V_t)\d \hat{B}_t, 
	\end{aligned}
	\right.$$
	where $\tilde{Z}_0=\tilde{Y}_0=Z_0=Y_0,\, V_t:=\tilde{Z}_t-\tilde{Y}_t$, and
	$$
	u_t:=\begin{cases}
		\frac{V_t}{\abs{V_t}},  &V_t\neq0,\\
		0,  &V_t=0.
	\end{cases}$$
	
	Similarly we have
	$$
	\begin{aligned}
		\d f(\abs{V_t})&\leq\left(-\kappa(\abs{V_t})\abs{V_t}f^{\prime}(\abs{V_t})+2f^{\prime\prime}(\abs{V_t})\lambda(V_t)^2\right)\d t\\
		&\ +\eta f^{\prime}(\abs{V_t})\Wcal_1\big(\Ecal_t^w(\tilde{Z}),\mu^*\big)\d t+2f^{\prime}(\abs{V_t})\lambda(V_t)e_t^\transpose\d B_t,
	\end{aligned}$$
	and then
	$$
	\begin{aligned}
		\frac{\d}{\d t}\e&[f(\abs{V_t})]\leq -\e[\abs{V_t}]+\rho(\delta)+\eta f^{\prime}(0)\e\big[\Wcal_1\big(\Ecal_t^w(\tilde{Z}),\mu^*\big)\big]\\
		&\leq -\e[\abs{V_t}]+\rho(\delta)+\eta f^{\prime}(0)\left(\e\big[\Wcal_1\big(\Ecal_t^w(\tilde{Z}),\Ecal_t^w(\tilde{Y})\big)\big]+\e\big[\Wcal_1\big(\Ecal_t^w(\tilde{Y}),\mu^*\big)\big]\right)\\
		&\leq -\e[\abs{V_t}]+\rho(\delta)+\eta f^{\prime}(0)\Big(\int_{0}^{1}\e[\abs{V_{ts}}]w_t(\d s)+\e\big[\Wcal_1\big(\Ecal_t^w(\tilde{Y}),\mu^*\big)\big]\Big)\\
		&= -\e[\abs{V_t}]+\eta f^{\prime}(0)\int_{0}^{1}\e[\abs{V_{ts}}]w_t(\d s)+\eta f^{\prime}(0)\e\big[\Wcal_1\big(\Ecal_t^w(\tilde{Y}),\mu^*\big)\big]+\rho(\delta).
	\end{aligned}$$
	Just as we did before, we define functions $\varphi(t):=\e[\abs{V_t}]$ and $\psi(t):= \e[f(\abs{V_t})]$, which satisfy $$\varphi,\,\psi\geq0,\, \varphi(0)=\psi(0)=0,\, \kappa_{\infty}^{-1}\varphi\leq\psi\leq f^{\prime}(0)\varphi,$$
	and take $\delta=\delta(t)$ such that $\rho(\delta(t))=O(t^{-1})$ as $t\to\infty$. 	
	By Corollary \ref{WeightedMarkovSDE}, for any $0<\varepsilon<\min\{\frac{\varepsilon_2}{2}(1-\frac{1}{q}),\frac{\varepsilon_2}{d}(1-\frac{1}{q})\}$, we have 
	$$
	\eta f^{\prime}(0)\e\big[\Wcal_1\big(\Ecal_t^w(\tilde{Y}),\mu^*\big)\big]=\eta f^{\prime}(0)\e\big[\Wcal_1\big(\Ecal_t^w(Y),\mu^*\big)\big]\leq M t^{-\varepsilon},$$ 
	where $M>0$ is a constant independent of $t$, $\varphi$ and $\psi$. Thus we get
	\begin{equation}\label{WeightedODI2.1}
		\frac{\d}{\d t}\psi(t)\leq -\varphi(t)+\eta f^{\prime}(0)\int_{0}^{1}\varphi(ts)w_t(\d s)+Mt^{-\varepsilon}
	\end{equation} 
	and therefore
	\begin{equation}\label{WeightedODI2.2}
		\frac{\d}{\d t}\psi(t)\leq -\frac{1}{f^{\prime}(0)}\psi(t)+\eta \kappa_{\infty}f^{\prime}(0)\int_{0}^{1}\psi(ts)w_t(\d s)+Mt^{-\varepsilon}.
	\end{equation}
	By \cite[Lemma 4.1]{du2023empirical}, for any $0<\varepsilon<\min\{\frac{\varepsilon_2}{d}(1-\frac{1}{q}),\frac{\varepsilon_2}{2}(1-\frac{1}{q}),\varepsilon_1\}$ we get
	$$
	\psi(t)\leq C t^{-\varepsilon}$$
	where the constant $C>0$ is independent of $t$ and $\psi$. By continuity of $\psi$, there exists a constant $\tilde{C}\in[C,\infty)$ independent of $t$ such that 
	$$
	\psi(t)\leq \tilde{C} (1\wedge t^{-\varepsilon})$$
	for all $t\geq0$.
		
	For $\overline{w}\in \mathbf{\Pi}_2(\varepsilon_2)$ and $0<\varepsilon<\min\{\frac{\varepsilon_2}{d}(1-\frac{1}{q}),\frac{\varepsilon_2}{2}(1-\frac{1}{q}),\varepsilon_1\}$, we have 
	$$
	\begin{aligned}
		\e\big[\Wcal_1\big(\Ecal_t^{\overline{w}}(Z),\Ecal_t^{\overline{w}}(Y)\big)\big]&\leq \int_{0}^{1}\varphi(ts)\overline{w}_t(\d s)\\
		&\leq \kappa_\infty\int_{0}^{1}\psi(ts)\overline{w}_t(\d s)\\
		&\leq \tilde{C}t^{-\varepsilon}\int_{0}^{1}t^{\varepsilon}\wedge s^{-\varepsilon}\overline{w}_t(\d s).
	\end{aligned}$$
	As $\varepsilon<\varepsilon_2$, we have $t^{\varepsilon}\wedge s^{-\varepsilon}\leq t^{\varepsilon_2}\wedge s^{-\varepsilon_2}$ for $t\geq1$, and thus 
	$\sup\limits_{t\geq0}\int_{0}^{1}t^{\varepsilon}\wedge s^{-\varepsilon}\overline{w}_t(\d s)<\infty.$
	Therefore, we have 
	\begin{equation}\label{weight_compare}
		 \e\big[\Wcal_1\left(\Ecal_t^{\overline{w}}(Z),\Ecal_t^{\overline{w}}(Y)\right)\big]=O(t^{-\varepsilon})
	\end{equation}
	for any $0<\varepsilon<\min\{\frac{\varepsilon_2}{d}(1-\frac{1}{q}),\frac{\varepsilon_2}{2}(1-\frac{1}{q}),\varepsilon_1\}$.
	
	Combining \eqref{weight_compare} with Corollary \ref{WeightedMarkovSDE}, we complete the proof of Theorem \ref{WeightedPathDependent}.
		
    \section{Wasserstein convergence for empirical measures of Markov processes}\label{appendix}
    By \cite[Remarks 2.8 and 5.1]{main}, for all solutions $Y=(Y_t)$ and $\tilde{Y}=(\tilde{Y}_t)$ of the Markovian SDE \eqref{sde2}, we have 
	\begin{equation}\label{contractive0}
		\Wcal_1\big(\Lcal(Y_t),\Lcal(\tilde{Y}_t)\big)\leq D\exp(-ct)\Wcal_1\big(\Lcal(Y_0),\Lcal(\tilde{Y}_0)\big),\quad \forall t\geq0,
	\end{equation}
	with constants $D=\kappa_{\infty}f^{\prime}(0)\geq1$ and $c=1/f^{\prime}(0)>0$. This means \eqref{sde2} is an exponentially contractive Markov process with a unique invariant measure $\mu^*\in\Pcal_1(\Real^d)$ identical to that of \eqref{sde1} (cf. \cite{wang2018distribution}). Hence in this section, we deal with empirical approximation to invariant measures of exponentially contractive Markov processes. 
		
    We consider a continuous-time Markov process $Y=(Y_t)_{t\geq0}$ on a complete separable metric space $(\Xcal, \rho)$, adapted to a complete filtered probability space $(\Omega, \mathscr{F},\mathscr{F}_t, \mathbb{P})$. Denote by $(P_t)_{t\geq0}$ the transition semigroup of $Y$ such that $\Lcal(Y_0)P_t:=\Lcal(Y_t)$, which satisfies the following two assumptions: 
            \begin{assum}\label{ContractiveMarkov}
			For the Markov transition semigroup $(P_t)_{t\geq0}$, there exist constants $D\geq1$ and $c>0$ such that 
			\begin{equation}\label{contractive}
				\Wcal_1(\mu_0 P_t,\nu_0 P_t)\leq D\exp(-ct)\Wcal_1(\mu_0,\nu_0)
			\end{equation}
			for all $\mu_0,\, \nu_0\in\Pcal_1(\Xcal)$ and $t\geq0$.
		\end{assum}
	 
		\begin{assum}\label{finite_measurable}
			For the Markov process $Y=(Y_t, P_t)_{t\geq0}$, we require $\limsup\limits_{t\to 0^+}\Wcal_1(\mu_0 P_t,\mu_0)<\infty$ for all $\mu_0\in\Pcal_1(\Xcal)$. Furthermore, we assume the measurability of $Y:\Omega\times[0,+\infty)\to\Xcal$ and $t\mapsto Y_t(\omega)$ for $\omega\in\Omega$ almost surely.
		\end{assum}
		
		\begin{re}\label{recontractive}
			To meet the conditions of Assumption \ref{ContractiveMarkov}, $(P_t)_{t\geq0}$ only needs to satisfy
			$$
			\Wcal_1(\delta_x P_t,\delta_y P_
			t)\leq D\exp(-ct)\Wcal_1(\delta_x,\delta_y)=D\exp(-ct)\rho(x,y)$$
			for all $x,\,y\in\Xcal$. Since for all $\mu,\,\nu\in\Pcal_1(\Xcal)$, we have 
			$$
			\begin{aligned}
				\Wcal_1(\mu P_t,\nu P_t)&= \sup\limits_{\abs{f}_1\leq1}\abs{(\mu P_t-\nu P_t)(f)}\\
				&=\sup\limits_{\abs{f}_1\leq1}\abs{(\mu-\nu)(P_t f)}\\
				&\leq \abs{P_t f}_1 \Wcal_1(\mu,\nu)\\
				&\leq D\exp(-ct)\Wcal_1(\mu,\nu),
			\end{aligned}$$
			where $P_t f(x):=\int_{\Xcal}f(y)P_t(x,\d y)$.
		\end{re}
		With the completeness of $(\Pcal_1(\Real^d), \Wcal_1)$, we can easily get the following result:
		\begin{prop}
			Under Assumptions \ref{ContractiveMarkov} and \ref{finite_measurable}, there exists a unique invariant measure $\mu^*\in\Pcal_1(\Xcal)$, i.e. $\mu^*P_t=\mu^*$ for all $t\geq0$, and 
			$$
			\Wcal_1(\mu_0 P_t,\mu^*)\leq D\exp(-ct)\Wcal_1(\mu_0,\mu^*)$$
			holds for all initial distributions $\mu_0\in\Pcal_1(\Xcal)$. 
		\end{prop}
	
		Now we let $\Xcal=\Real^d$, and $\rho$ be the Euclidean metric. We denote by $\mu_0\in\Pcal_1(\Real^d)$ the initial distribution of $(Y_t, P_t)$, and by  
		$$\mu_t:=\Ecal_t^{\overline{w}}(Y)=\int_{0}^{1}\delta_{Y_{ts}}\overline{w}_t(\d s), \quad \forall t>0$$
        the weighted empirical measures, where $\overline{w}\in\mathbf{\Pi}_2(\varepsilon)$ for some $\varepsilon\in(0,1]$ (recall the definition of $\mathbf{\Pi}_2(\varepsilon)$ in \eqref{weight}). Our target in this section is to estimate the convergence rate of $\e[\Wcal_1(\mu_t,\mu^*)]$ as $t\to\infty$. Thus, we require another assumption:
		\begin{assum}\label{LqMarkov}
			For the Markov process $(Y_t, P_t)_{t\geq0}$, there exist constants $q>1$ and $M_q\in(0,\infty)$ such that
			$$
			\sup\limits_{t\geq0}(\e[\abs{Y_t}^q])^{1/q}\leq M_q.$$
		\end{assum}
		
		\begin{re}
			It's easy to prove that the invariant measure $\mu^*$ also satisfies 
			$$
			\Big(\int_{\Real^d}\abs{x}^q\mu^*(\d x)\Big)^{1/q}\leq M_q.$$
		\end{re}
		
		\begin{re}
			If $q=\infty$, then Assumption \ref{LqMarkov} suggests $Y$ is supported on a compact set in $\Real^d$.
		\end{re}
	
		For $a,\,b\geq0$, we write $a\lesssim b$ if there exists a positive constant $C=C(d,D,c,\overline{w},\varepsilon)>0$ (independent of $t,\,q$ and $M_q$) such that $a\leq Cb$. We are going to prove the following result: 
		\begin{theo}\label{WeightedMarkovConvergence}
			For a continuous-time Markov process $(Y_t, P_t)_{t\geq0}$ that satisfies Assumptions \ref{ContractiveMarkov}, \ref{finite_measurable} and \ref{LqMarkov}, and a weight family $\overline{w}\in\mathbf{\Pi}_2(\varepsilon)$ for some $\varepsilon\in(0,1]$, we have  
			$$
			\e[\Wcal_1(\mu_t,\mu^*)]\lesssim M_q
			\begin{cases}
				\Big(\frac{(\log\hat{t}\,)^{d-2+\frac{1}{d}}}{\hat{t}^{\frac{1}{d}}}\Big)^{1-\frac{1}{q}},  &d\geq 3,\\
				\Big(\frac{\log \hat{t}}{\sqrt{\hat{t}}}\Big)^{1-\frac{1}{q}},  &d=2,\\
				\bigg(\sqrt{\frac{\log \hat{t}}{\hat{t}}}\bigg)^{1-\frac{1}{q}},  &d=1,
			\end{cases}$$
			for all large enough $t>0$, where $\hat{t}=t^{\varepsilon}$.
		\end{theo}
		This theorem can directly lead to the SDE version:
		\begin{coro}\label{WeightedMarkovSDE}
			Under Assumptions \ref{dissi}, \ref{weaki} and \ref{Lq}, for any $\varepsilon\in(0,1],\, \overline{w}\in\mathbf{\Pi}_2(\varepsilon)$, and any solution $Y=(Y_t)_{t\geq0}$ of the Markovian SDE \eqref{sde2}, we have
			$$
			\e\big[\Wcal_1\big(\Ecal_t^{\overline{w}}(Y),\mu^*\big)\big]=O\left(t^{-\varepsilon}\right),$$
			where $0<\varepsilon<\min\{\frac{\varepsilon}{d}(1-\frac{1}{q}),\frac{\varepsilon}{2}(1-\frac{1}{q})\}$.
		\end{coro}
            Without loss of generality, we assume $M_q=1$ (otherwise we replace the norm $\abs{\cdot}$ with $\abs{\cdot}/M_q$) throughout the rest of Section \ref{appendix}, which simplifies the proof of Theorem \ref{WeightedMarkovConvergence}.
		
        Following the ideas of \cite{kloeckner2020empirical,riekert2022convergence}, we implement Kantorovich duality $\sup\limits_{\abs{f}\leq1}\abs{\mu_t(f)-\mu(f)}$ to characterize $\Wcal_1(\mu_t,\mu^*)$; for a Lipschitz function $f$, we estimate $\abs{\mu_t(f)-\mu(f)}$ by approximating $f$ with its truncated Fourier series. Thus we need to study the Fourier basis function $$\mathbf{e}_k(x):=\exp(\pi i\langle k,x\rangle)$$ for all $k\in\ZZ^d,\, x\in\Real^d$, whose Lipschitz constant $\abs{\nabla \mathbf{e}_k}=\pi\abs{k}$ explodes as $\abs{k}\to\infty$. By contrast, for $\alpha\in(0,1]$ we have (cf. ~\cite[Lemma 4.2]{kloeckner2020empirical})
        \begin{equation}\label{fourier_base}
            \abs{\mathbf{e}_k}_{\alpha}\lesssim \norm{k}_{\infty}^{\alpha},
        \end{equation}
	which means the growth rate of $\abs{\mathbf{e}_k}_{\alpha}$ can be controlled by optimizing the value of $\alpha$. Therefore, we delve into the $\alpha$-${\rm H\ddot{o}lder}$ constant $\abs{\cdot}_\alpha$, acquiring Lemma \ref{lealpha02} and then Proposition \ref{propalpha_weighted}:
		\begin{lemma}\label{lealpha02}
			For $\alpha\in(0,1]$, let $W_{\alpha}$ be the 1-Wasserstein distance with respect to the revised metric $d_{\alpha}(x,y)=\abs{x-y}^{\alpha}$. Then we have
			\begin{equation}\label{markov_contraction}
				W_{\alpha}(\mu_0 P_t,\nu_0 P_t)\leq D^{\alpha}\exp(-\alpha ct)W_{\alpha}(\mu_0,\nu_0)
			\end{equation}
			for all $\mu_0,\,\nu_0\in\Pcal_1(\Real^d)$.\\
			Moreover, for all bounded $\alpha$-${\rm H\ddot{o}lder}$ functions $f:\Real^d\to\mathbb{C}$ and $s\geq t\geq 0$, we have
			\begin{align}
			\label{alpha_align01}	\abs{\e[f(Y_t)]-\mu^*(f)}&\lesssim \abs{f}_{\alpha}\exp(-\alpha ct)\\
			\label{alpha_align02}	\abs{{\rm Cov}(f(Y_s), f(Y_t))}&\lesssim \norm{f}_{\infty}\abs{f}_{\alpha}\exp(-\alpha c(s-t)).
			\end{align} 
		\end{lemma}
	
		\begin{proof}
			By Jensen's inequality, for $x,\, y\in\Real^d,\, \mu_0=\delta_x,\,\nu_0=\delta_y$ we get
			$$
			\begin{aligned}
				W_{\alpha}(\mu_0 P_t,\nu_0 P_t)&\leq (\Wcal_1(\mu_0 P_t,\nu_0 P_t))^{\alpha}\\
				&\leq D^{\alpha} \exp(-\alpha ct) (\Wcal_1(\mu_0,\nu_0))^{\alpha}\\
				&=D^{\alpha} \exp(-\alpha ct) \abs{x-y}^{\alpha}\\
				&=D^{\alpha} \exp(-\alpha ct)W_{\alpha}(\mu_0,\nu_0).
			\end{aligned}$$
			 By Remark \ref{recontractive}, we finish the proof of \eqref{markov_contraction}.
			
			Therefore, for an $\alpha$-${\rm H\ddot{o}lder}$ function $f:\Real^d\to\mathbb{C}$, we have 
			$$
			\begin{aligned}
				\abs{\e[f(Y_t)]-\mu^*(f)}&\leq \abs{f}_{\alpha}W_{\alpha}(\mu_0 P_t,\mu^*)\\
				&\leq D^{\alpha}\abs{f}_{\alpha} \exp(-\alpha ct)W_{\alpha}(\mu_0,\mu^*)\\
				&\lesssim \abs{f}_{\alpha}\exp(-\alpha ct),
			\end{aligned}$$
            which yields \eqref{alpha_align01}. Here we use 
			$$W_{\alpha}(\mu_0,\mu^*)\leq \e[\abs{Y_0-Y^*}^{\alpha}]\leq \norm{Y_0-Y^*}_{L^q}^{\alpha}\leq (2M_q)^\alpha=2^{\alpha}\lesssim1,$$
			where $Y^*$ is a random variable with distribution $\mu^*$.
			
			For any $x\in\Real^d$, we let $Y_0=x$, and get 
			$$\abs{P_t f(x)-\mu^*(f)}=\abs{\e[f(Y_t)]-\mu^*(f)}\lesssim \abs{f}_{\alpha}\exp(-\alpha ct),$$ 
			i.e. 
			$$
			\norm{P_t f-\mu^*(f)}_{\infty}\lesssim \abs{f}_{\alpha}\exp(-\alpha ct).$$
			Without loss of generality, we let $\mu^*(f)=0$ (otherwise we replace $f$ with $\tilde{f}=f-\mu^*(f)$, and get ${\rm Cov}(f(Y_s), f(Y_t))={\rm Cov}(\tilde{f}(Y_s), \tilde{f}(Y_t))$, $\abs{f}_{\alpha}=\abs{\tilde{f}}_{\alpha}$, and $\norm{\tilde{f}}_{\infty}\leq 2\norm{f}_{\infty}\lesssim \norm{f}_{\infty}$). Thus we have
			$$
			\begin{aligned}
				\abs{{\rm Cov}(f(Y_s), f(Y_t))}&\leq \abs{\e\Big[f(Y_s)\overline{f(Y_t)}\Big]}+\abs{\e[f(Y_s)]\overline{\e[f(Y_t)]}}\\
				&=\abs{\e\Big[\e[f(Y_s)|Y_t]\overline{f(Y_t)}\Big]}+\abs{\e[f(Y_s)]\e[f(Y_t)]}\\
				&=\abs{\e\Big[P_{s-t}f(Y_t)\overline{f(Y_t)}\Big]}+\abs{\e[f(Y_s)]\e[f(Y_t)]}\\
				&\leq \norm{P_{s-t}f}_{\infty}\e[\abs{f(Y_t)}]+\abs{\e[f(Y_s)]}\e[\abs{f(Y_t)}]\\
				&\leq \norm{P_{s-t}f}_{\infty}\norm{f}_{\infty}+\norm{f}_{\infty}\abs{\e[f(Y_s)]}\\
				&\lesssim \norm{f}_{\infty}\abs{f}_{\alpha}\exp(-\alpha c(s-t))+\norm{f}_{\infty}\abs{f}_{\alpha}\exp(-\alpha cs)\\
				&\lesssim \norm{f}_{\infty}\abs{f}_{\alpha}\exp(-\alpha c(s-t)),
			\end{aligned}$$
			and complete the proof.
		\end{proof}

		\begin{prop}\label{propalpha_weighted}
			\begin{equation}
				\e[\abs{(\mu_t-\mu^*)(f)}^2]\lesssim \frac{(\norm{f}_{\infty}+\abs{f}_{\alpha})\abs{f}_{\alpha}}{(\alpha t)^{\varepsilon}}
			\end{equation}
			for all bounded $\alpha$-${\rm H\ddot{o}lder}$ functions $f:\Real^d\to\mathbb{C}$ and $t>0$.
		\end{prop}
		\begin{proof}
			By Fubini's Theorem, we have 
			$$
			\begin{aligned}
				\e&[\abs{(\mu_t-\mu^*)(f)}^2]=\e\Big[\abs{\int_{0}^{1}(f(Y_{ts})-\mu^*(f))\overline{w}_t(\d s)}^2\Big]\\
				&=\e\Big[\int_{0}^{1}(f(Y_{ts_1})-\mu^*(f))\overline{w}_t(\d s_1)\int_{0}^{1}\overline{(f(Y_{ts_2})-\mu^*(f))}\overline{w}_t(\d s_2)\Big]\\
				&=\e\Big[\iint_{[0,1]\times [0,1]}(f(Y_{ts_1})-\mu^*(f))\overline{(f(Y_{ts_2})-\mu^*(f))}\overline{w}_t(\d s_1) \overline{w}_t(\d s_2)\Big]\\
				&=\iint_{[0,1]\times [0,1]}\e\Big[(f(Y_{ts_1})-\mu^*(f))\overline{(f(Y_{ts_2})-\mu^*(f))}\Big]\overline{w}_t(\d s_1) \overline{w}_t(\d s_2)\\
				&=\iint_{[0,1]\times [0,1]} {\rm Cov}(f(Y_{ts_1}),f(Y_{ts_2}))\overline{w}_t(\d s_1) \overline{w}_t(\d s_2)\\
				&\ +\iint_{[0,1]\times [0,1]}(\e[f(Y_{ts_1})]-\mu^*(f))\overline{(\e[f(Y_{ts_2})]-\mu^*(f))}\overline{w}_t(\d s_1) \overline{w}_t(\d s_2).
                \end{aligned}$$
                Applying Lemma \ref{lealpha02}, we have
                $$
                \begin{aligned}
				\e&[\abs{(\mu_t-\mu^*)(f)}^2]\\
                &\leq \iint_{[0,1]\times [0,1]} \abs{{\rm Cov}(f(Y_{ts_1}),f(Y_{ts_2}))}\overline{w}_t(\d s_1) \overline{w}_t(\d s_2)\\
				&\ +\iint_{[0,1]\times [0,1]}\abs{\e[f(Y_{ts_1})]-\mu^*(f)}\abs{\e[f(Y_{ts_2})]-\mu^*(f)}\overline{w}_t(\d s_1) \overline{w}_t(\d s_2)\\
				&\lesssim (\norm{f}_{\infty}+\abs{f}_{\alpha})\abs{f}_{\alpha} \iint_{[0,1]\times [0,1]} \exp(-\alpha ct\abs{s_1-s_2})\overline{w}_t(\d s_1) \overline{w}_t(\d s_2)\\
				&= \frac{(\norm{f}_{\infty}+\abs{f}_{\alpha})\abs{f}_{\alpha}}{(\alpha t)^{\varepsilon}} \iint_{[0,1]\times [0,1]} (\alpha t)^{\varepsilon}\exp(-\alpha ct\abs{s_1-s_2})\overline{w}_t(\d s_1) \overline{w}_t(\d s_2).
			\end{aligned}$$
			We notice that 
			$$
			(\alpha t)^{\varepsilon}\exp(-\alpha ct\abs{s_1-s_2})\leq (\alpha t)^{\varepsilon}\wedge (ce\abs{s_1-s_2})^{-\varepsilon}\lesssim t^{\varepsilon}\wedge \abs{s_1-s_2}^{-\varepsilon}.$$
			Then by $\overline{w}\in\mathbf{\Pi}_2(\varepsilon)$, we get
			$$
			\e[\abs{(\mu_t-\mu^*)(f)}^2]\lesssim \frac{(\norm{f}_{\infty}+\abs{f}_{\alpha})\abs{f}_{\alpha}}{(\alpha t)^{\varepsilon}},$$
            which completes the proof.
		\end{proof}
		
		Now we divide $\Real^d$ into a compact set $K:=[-R,R]^d$ and its complementary set $K^c:=\Real^d\setminus K$, where the value of $R\in(0,\infty)$ will be optimized later. For any $f:\Real^d\to\mathbb{C}$ such that $\abs{f}_1\leq1$, we assume $f(0)=0$ by translating it, and by triangle inequality we get
		\begin{lemma}\label{tri}
			For any bounded measurable function $h: \Real^d\to\mathbb{C}$, we have 
			$$
			\abs{\mu_t(f)-\mu^*(f)}\leq 2\norm{f-h}_{L^\infty(K)}+\abs{\mu_t(h)-\mu^*(h)}+\int_{K^c}(\abs{f}+\abs{h})(\d\mu_t+\d\mu^*).$$
		\end{lemma} 
		We extend $f|_K$ to $\tilde{f}:[-2R,2R]^d\to\mathbb{C}$, which satisfies $\abs{\tilde{f}}_1\leq1$ and periodic boundary conditions on $[-2R,2R]^d$, and define $g:[-1,1]^d\to\mathbb{C}$ such that $g(x):=\tilde{f}(2Rx)$. Then $g$ satisfies periodic boundary conditions and $\abs{g}_1\leq 2R$. For a given $J\in\ZZ_{>0}$, we consider the truncation of $g$'s Fourier series $F_J(g):= \sum\limits_{\norm{k}_{\infty}\leq J} \hat{g_k}\mathbf{e}_k$, and get (by \cite[Theorem 4.4]{fourier2}) 
		$$
		\norm{F_J(g)-g}_{L^\infty([-1,1]^d)}\leq CR\frac{(\log J)^d}{J},$$
		where $C=C(d)>0$ is a constant dependent only on the dimension $d$. We then write $T_J(f)(x):=F_J(g)(x/2R)$, and get 
		$$
		\norm{T_J(f)-f}_{L^\infty([-R,R]^d)}\leq \norm{T_J(f)-\tilde{f}}_{L^\infty([-2R,2R]^d)}\lesssim R\frac{(\log J)^d}{J}.$$
  
		Next, we let $h=T_J(f)$ in Lemma \ref{tri} to estimate $\abs{(\mu_t-\mu^*)(f)}$. 
		Since $f(0)=0,\, \abs{f}_1\leq 1$, we have $\abs{f(x)}\leq \abs{x}$ for all $x\in\Real^d$. By ${\rm H\ddot{o}lder}$'s inequality, we have 
		$$
		\int_{K^c}\abs{f}\d\mu^*\leq \int_{K^c}\abs{x}\d\mu^*\leq \Big(\int_{K^c}\abs{x}^q\d\mu^*\Big)^{\frac{1}{q}}\mu^*(K^c)^{1-\frac{1}{q}}\leq \mu^*(K^c)^{1-\frac{1}{q}}.$$
		Also, by Chebyshev's inequality we have 
		\begin{equation}\label{cheby}
		    \begin{aligned}
			\mu^*(K^c)&\leq \mu^*(\abs{x}>R)\leq R^{-q}\int_{\Real^d}\abs{x}^q\d\mu^*\leq R^{-q},\\
			\p(Y_t\in K^c)&\leq \p(\abs{Y_t}>R)\leq R^{-q}\e[\abs{Y_t}^q]\leq R^{-q},
		\end{aligned}
		\end{equation}
		and thus we get $\int_{K^c}\abs{f}\d\mu^*\leq R^{1-q}$.
		By  periodicity of $T_J(f)$, we have
		$$
		\norm{T_J(f)}_{L^{\infty}(K^c)}=\norm{T_J(f)}_{L^{\infty}([-2R,2R]^d)}\lesssim R\frac{(\log J)^d}{J}+\norm{\tilde{f}}_{L^{\infty}([-2R,2R]^d)}\lesssim R\frac{(\log J)^d}{J}+R.$$  
		Therefore, 
		$$
		\begin{aligned}
			\int_{K^c}\abs{T_J(f)}\d\mu^*&\lesssim \mu^*(K^c)(R\frac{(\log J)^d}{J}+R)\leq R\frac{(\log J)^d}{J}+R^{1-q},\\
			\int_{K^c}\abs{T_J(f)}\d\mu_t&\lesssim \mu_t(K^c)(R\frac{(\log J)^d}{J}+R)\leq R\frac{(\log J)^d}{J}+R\mu_t(K^c).        
		\end{aligned}$$
		We define $\mathbf{e}_k^R(x):=\mathbf{e}_k(\frac{x}{2R})$, and get
		$$
		\begin{aligned}
			\abs{(\mu_t-\mu^*)(T_J(f))}&\leq \sum\limits_{k\in\ZZ^d,\norm{k}_{\infty}\leq J }\abs{\hat{g_k}}\abs{(\mu_t-\mu)(\mathbf{e}_k^R)}\\
			&=0+\sum\limits_{0<\norm{k}_{\infty}\leq J }\abs{\hat{g_k}}\abs{(\mu_t-\mu^*)(\mathbf{e}_k^R)}\\
			&\leq \Big(\sum\limits_{0<\norm{k}_{\infty}\leq J }\abs{\hat{g_k}}^2 \norm{k}_{\infty}^2 \Big)^{\frac{1}{2}}\Big(\sum\limits_{0<\norm{k}_{\infty}\leq J}\frac{\abs{(\mu_t-\mu^*)(\mathbf{e}_k^R)}^2}{\norm{k}_{\infty}^2}\Big)^{\frac{1}{2}}\\
			&\lesssim \norm{g}_{H^1([-1,1]^d)}\Big(\sum\limits_{0<\norm{k}_{\infty}\leq J}\frac{\abs{(\mu_t-\mu)(\mathbf{e}_k^R)}^2}{\norm{k}_{\infty}^2}\Big)^{\frac{1}{2}}\\
			&\lesssim 2R\Big(\sum\limits_{0<\norm{k}_{\infty}\leq J}\frac{\abs{(\mu_t-\mu)(\mathbf{e}_k^R)}^2}{\norm{k}_{\infty}^2}\Big)^{\frac{1}{2}},
		\end{aligned}$$ 
		where $H^1$ denotes the Sobolev space.
		From the above, we obtain the following intermediate result:
		\begin{lemma}\label{Linter}
			For any Lipschitz function $f:\Real^d\to\mathbb{C}$ such that $\abs{f}_1\leq1$, and $R>0,\,J\in\mathbb{Z}_{>0}$, we have
			\begin{equation}\label{Linter_equation}
                    \begin{aligned}
                        \abs{\mu_t(f)-\mu^*(f)}&\lesssim R\frac{(\log J)^d}{J}+R^{1-q}+\int_{K^c}\abs{x}\d\mu_t\\
            &\ +R\mu_t(K^c)+R\Big(\sum\limits_{0<\norm{k}_{\infty}\leq J}\frac{\abs{(\mu_t-\mu^*)(\mathbf{e}_k^R)}^2}{\norm{k}_{\infty}^2}\Big)^{\frac{1}{2}}.
                    \end{aligned}
			\end{equation}	
		\end{lemma}
		Taking supreme upper bound w.r.t $f$, and then expectation on both sides of \eqref{Linter_equation}, by Kantorovich duality and \eqref{cheby} we have 
		$$
		\begin{aligned}
			\e[\Wcal_1(\mu_t,\mu^*)]&\lesssim R\frac{(\log J)^d}{J}+R^{1-q}+\int_{0}^{1}\e[\abs{Y_{ts}}\cdot\mathbb{I}_{K^c}(Y_{ts})]\overline{w}_t(\d s)\\
			&\ +R\int_{0}^{1}\p(Y_{ts}\in K^c)\overline{w}_t(\d s)+R\e\bigg[\Big(\sum\limits_{0<\norm{k}_{\infty}\leq J}\frac{\abs{(\mu_t-\mu^*)(\mathbf{e}_k^R)}^2}{\norm{k}_{\infty}^2}\Big)^{\frac{1}{2}}\bigg]\\
			&\lesssim R\frac{(\log J)^d}{J}+R^{1-q}+\int_{0}^{1}\p(Y_{ts}\in K^c)^{1-\frac{1}{q}}\overline{w}_t(\d s)\\
                &\ +R\e\bigg[\Big(\sum\limits_{0<\norm{k}_{\infty}\leq J}\frac{\abs{(\mu_t-\mu^*)(\mathbf{e}_k^R)}^2}{\norm{k}_{\infty}^2}\Big)^{\frac{1}{2}}\bigg]\\
			&\lesssim R\frac{(\log J)^d}{J}+R^{1-q}+R\Big(\sum\limits_{0<\norm{k}_{\infty}\leq J}\frac{\e[\abs{(\mu_t-\mu^*)(\mathbf{e}_k^R)}^2]}{\norm{k}_{\infty}^2}\Big)^{\frac{1}{2}}.
		\end{aligned}$$
		Without loss of generality, we assume $R>1/2$ (in fact, we are going to let $R\to\infty$ as $t\to\infty$), so that $\abs{\mathbf{e}_k^R}_{\alpha}=(1/{2R})^{\alpha}\abs{\mathbf{e}_k}_{\alpha}\leq \abs{\mathbf{e}_k}_{\alpha}$. Applying \eqref{fourier_base} and Proposition \ref{propalpha_weighted}, we get
		$$
		\begin{aligned}
			\Big(\sum\limits_{0<\norm{k}_{\infty}\leq J}\frac{\e[\abs{(\mu_t-\mu^*)(\mathbf{e}_k^R)}^2]}{\norm{k}_{\infty}^2}\Big)^{\frac{1}{2}}&\lesssim \Big(\sum\limits_{0<\norm{k}_{\infty}\leq J}\frac{\norm{k}_{\infty}^{2\alpha-2}}{(\alpha t)^{\varepsilon}}\Big)^{\frac{1}{2}}\\
			&\lesssim \Big(\sum\limits_{l=1}^J\frac{l^{d+2\alpha-3}}{(\alpha t)^{\varepsilon}}\Big)^{\frac{1}{2}}.
		\end{aligned}$$
		Taking $\alpha=1/\log J$, we have $J^{2\alpha}=\exp(2)\lesssim 1$, and therefore
		$$
		\Big(\sum\limits_{0<\norm{k}_{\infty}\leq J}\frac{\e[\abs{(\mu_t-\mu^*)(\mathbf{e}_k^R)}^2]}{\norm{k}_{\infty}^2}\Big)^{\frac{1}{2}}\lesssim \sqrt{\frac{(\log J)^{\varepsilon}}{\hat{t}}}\Big(\sum\limits_{l=1}^J l^{d-3}\Big)^{1/2}\lesssim \sqrt{\frac{\log J}{\hat{t}}}\Big(\sum\limits_{l=1}^J l^{d-3}\Big)^{1/2},$$
		where $\hat{t}=t^{\varepsilon}$.
		
		For $d\geq3$, we have $\sum\limits_{l=1}^J l^{d-3}\leq J^{d-2}$, so that 
		$$
		\e[\Wcal_1(\mu_t,\mu^*)]\lesssim R\left(\frac{(\log J)^d}{J}+J^{\frac{d-2}{2}}\sqrt{\frac{\log J}{\hat{t}}}\right)+R^{1-q}.$$
		Taking $R=(q-1)^{\frac{1}{q}}\left(\frac{(\log J)^d}{J}+J^{\frac{d-2}{2}}\sqrt{\frac{\log J}{\hat{t}}}\right)^{-\frac{1}{q}}$, we have 
		$$
		\e[\Wcal_1(\mu_t,\mu^*)]\lesssim\left(\frac{(\log J)^d}{J}+J^{\frac{d-2}{2}}\sqrt{\frac{\log J}{\hat{t}}}\right)^{1-\frac{1}{q}}.$$
		Let $J=\floor{(\log \hat{t})^{2-\frac{1}{d}} \hat{t}^{\frac{1}{d}}}$, and for large enough $t$ we have 
		$$
		\e[\Wcal_1(\mu_t,\mu^*)]\lesssim\left(\frac{(\log\hat{t}\,)^{d-2+\frac{1}{d}}}{\hat{t}^{\frac{1}{d}}}\right)^{1-\frac{1}{q}}.$$
		
		For $d=2$, we have $\sum\limits_{j=1}^J j^{d-3}\lesssim \log J$, so that
		$$
		\e[\Wcal_1(\mu_t,\mu^*)]\lesssim R\left(\frac{(\log J)^2}{J}+\frac{\log J}{\sqrt{\hat{t}}}\right)+R^{1-q}.$$
		Taking $R=(q-1)^{\frac{1}{q}}\left(\frac{(\log J)^2}{J}+\frac{\log J}{\sqrt{\hat{t}}}\right)^{-\frac{1}{q}}$, we have 
		$$
		\e[\Wcal_1(\mu_t,\mu^*)]\lesssim\left(\frac{(\log J)^2}{J}+\frac{\log J}{\sqrt{\hat{t}}}\right)^{1-\frac{1}{q}}.$$
		Let $J=\floor{\hat{t}^{\frac{1}{2}}\log\hat{t}}$, and for large enough $t$ we have 
		$$
		\e[\Wcal_1(\mu_t,\mu^*)]\lesssim\left(\frac{\log \hat{t}}{\sqrt{\hat{t}}}\right)^{1-\frac{1}{q}}.$$
		
		Finally for $d=1$, $\sum\limits_{j=1}^J j^{d-3}$ is bounded, so that 
		$$
		\e[\Wcal_1(\mu_t,\mu^*)]\lesssim R\left(\frac{\log J}{J}+\sqrt{\frac{\log J}{\hat{t}}}\right)+R^{1-q}.$$
		Taking $R=(q-1)^{\frac{1}{q}}\left(\frac{\log J}{J}+\sqrt{\frac{\log J}{\hat{t}}}\right)^{-\frac{1}{q}}$, we have
		$$
		\e[\Wcal_1(\mu_t,\mu^*)]\lesssim\left(\frac{\log J}{J}+\sqrt{\frac{\log J}{\hat{t}}}\right)^{1-\frac{1}{q}}.$$
		Let $J=\floor{\sqrt{\hat{t}\log \hat{t}}}$, and for large enough $t$ we have 
		$$
		\e[\Wcal_1(\mu_t,\mu^*)]\lesssim\left(\sqrt{\frac{\log \hat{t}}{\hat{t}}}\right)^{1-\frac{1}{q}}.$$
		This completes the proof of Theorem \ref{WeightedMarkovConvergence}. 
		
	\section{Numerical experiments}\label{Experiments}
	We consider the following Curie--Weiss mean-field lattice model on $\Real^1$:
	\begin{equation}\label{Curie}
		\d X_t= \left[-\beta(X_t^3-X_t)+\beta K\e[X_t]\right]\d t+\d B_t,
	\end{equation}
	with $\beta=1,\, K>0$. Evidently, \eqref{Curie} is an example of \eqref{sde1} such that $b(x,\mu)=-\beta(x^3-x)+\beta K\int_{\Real^1}u\mu(\d u)$ for $x\in\Real^1,\,\mu\in\Pcal_1(\Real^1)$. Although \eqref{Curie} doesn't satisfy the monotonicity condition in \cite{du2023empirical}, it fulfills the requirements of Assumption \ref{dissi} with $\kappa(r)=\beta(\frac{r^2}{4}-1)$, and Assumption \ref{weaki} if $\frac{1}{K}>\sqrt{2\pi\beta \exp(\beta)}\Phi(\sqrt{\beta})$, where $\Phi(\cdot)$ denotes the distribution function of a standard Gauss variable (cf. \cite{main}). By symmetry, the invariant measure $\mu^*$ of \eqref{Curie} must be symmetric about the origin, and thus $\int_{\Real^1}u\mu^*(\d u)=0$. Hence, we consider the following Markovian SDE with an identical invariant measure:
	\begin{equation}\label{CurieMarkov}
		\d Y_t= -\beta(Y_t^3-Y_t)\d t+\d B_t,
	\end{equation}
	and, by solving a Fokker--Planck equation, get the density function $p^*(\cdot)$ of $\mu^*$:
	$$
	p^*(x)=\frac{1}{C}\exp\big(-2\beta(\frac{x^4}{4}-\frac{x^2}{2})\big),$$
	where $C=\int_{\Real^1}\exp\big(-2\beta(\frac{x^4}{4}-\frac{x^2}{2})\big)\dx$.      

	Then we approximate \eqref{Curie} and $\mu^*$ with the following path-dependent process:
	\begin{equation}\label{CuriePath}
		\d Z_t= \Big[-\beta(Z_t^3-Z_t)+\beta K\int_0^1Z_{ts}\d s\Big]\d t+\d B_t.
	\end{equation} 
	 We notice that the drift term is superlinear, and therefore we need tamed Euler approximations (cf. \cite{10.1214/ECP.v18-2824}) to avoid explosion by replacing $b(\cdot,\cdot)$ with 
	\begin{equation}
		b^{(n_0,\alpha)}(\cdot,\cdot):=\frac{b(\cdot,\cdot)}{1+n_0^{-\alpha}\abs{b(\cdot,\cdot)}},
	\end{equation}
	where $\alpha\in(0,\frac{1}{2}]$ and $n_0\geq1$. Here we take $n_0=10000$ and $\alpha=0.1$.
	For all $k\geq0,\,1\leq n\leq N$, we generate $\hat{Z}_{(k+1)\Delta t}$ with the following recursive formula:
	\begin{equation}
		\hat{Z}_{(k+1)\Delta t}=\frac{\hat{b}_k(\hat{Z})}{1+n_0^{-\alpha}\abs{\hat{b}_k(\hat{Z})}}\Delta t+\xi^{(n)}_k,
	\end{equation}
	where $\{\xi_k^{(n)}\}_{k\geq0}^{1\leq n\leq N}$ are independent Gauss variables with distribution $\mathcal{N}(0,\Delta t)$, and
	$$\hat{b}_k(\hat{Z})=-\beta\big((\hat{Z}_{k\Delta t})^3-\hat{Z}_{k\Delta t}\big)+\beta K\frac{1}{k+1}\sum_{j=0}^{k}\hat{Z}_{j\Delta t}.$$ 
	
	In our numerical experiment, we take the step length $\Delta t=0.1$, and set $N=1000$ independent sample paths with initial values $\hat{Z}^{(n)}_0\sim\mathcal{N}(0,1)$ for all $1\leq n\leq N$. We run the above iteration process $5\times10^4$ times, and record the expectation of Wasserstein distances $\e\big[\Wcal_1(\Ecal_t(\hat{Z}),\mu^*)\big]:=\frac{1}{N}\sum_{n=1}^N\Wcal_1(\Ecal_t(\hat{Z}^{(n)}),\mu^*)$ for $t=100k\Delta t,\ 1\leq k\leq 500$. Also, we conduct six independent experiments altogether with $K=\frac{1}{5}j$ for $j=1,2,...,6$ respectively and compare them with each other. Finally, we obtain the following double logarithmic graph, where $x$ and $y$ axes represent the steps of iteration and expectation of Wasserstein distances respectively:    
	\begin{figure}[H]
		\centering
		\includegraphics{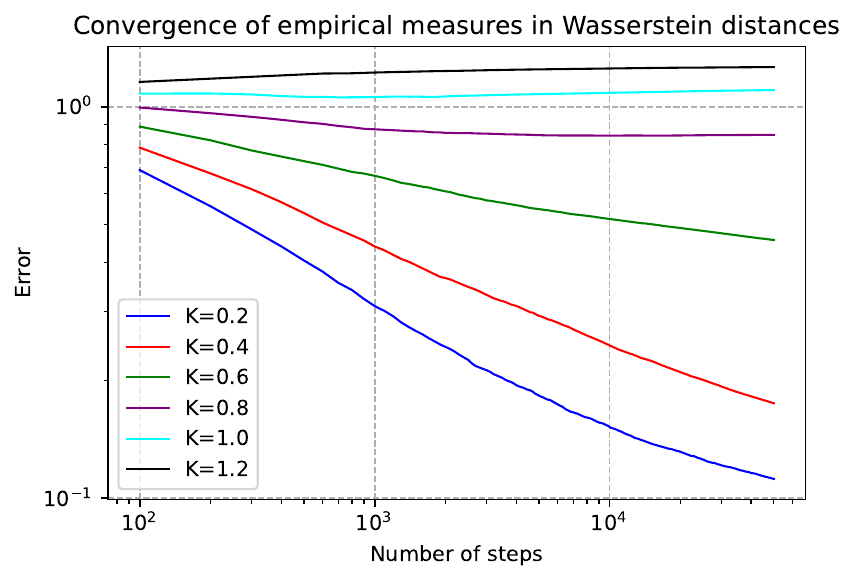}
		\caption{Expectation of Wasserstein distances between simulated empirical measures $\Ecal_t(\hat{Z})$ and invariant measure $\mu^*$ of Curie Weiss mean-field lattice model \eqref{Curie}.}
	\end{figure}
	Theoretically, we need $0<K<\big(\sqrt{2\pi \exp(1)}\Phi(1)\big)^{-1}\approx0.2876$ to satisfy Assumption \ref{weaki}. From the above graph, however, we see despite the fact that only the case $K=0.2$ satisfies Assumption \ref{weaki}, $K=0.4$ and $K=0.6$ cases also show some convergence behavior, though at lower rates. As $K$ increases, empirical measures begin to diverge from $\mu^*$ at $K=0.8$, and divergence becomes more noticeable at $K=1.0$ and $K=1.2$. These phenomena verify our theoretical results and support the intuitive hypothesis that the weaker the interactions between particles, the better the convergence performance of a McKean--Vlasov process, since it gets closer to a Markovian SDE. Furthermore, experiments go beyond our expectations, as our algorithm to approximate McKean--Vlasov dynamics still works well under conditions weaker than the theoretical assumptions, which highlights the effectiveness of the algorithm and calls for further research to improve the theory. 
\bibliographystyle{plain}
\bibliography{myref.bib}    
\end{document}